\documentclass[a4paper,11pt]{amsart}
\usepackage{amsmath,amssymb,amsthm} 
\usepackage{graphics} 

\usepackage{epsfig}
\usepackage{color}

\numberwithin{equation}{section}
\newtheorem{theorem}{Theorem}[section]
\newtheorem{lemma}[theorem]{Lemma}

\newtheorem{cor}[theorem]{Corollay}

\theoremstyle{definition}

\newtheorem{remark}[theorem]{Remark}

\newcommand{\R}{\mathbb{R}}

\newcommand{\eps}{\varepsilon}

\newcommand{\be}{\begin{equation}}
\newcommand{\ee}{\end{equation}}

\newcommand\lt{\left}
\newcommand\rt{\right}

\def\Div{\textup{div}}
\def\HH{\mathcal{H}}

\def\tr{{\rm tr}}
\def\co{{\rm co}}

\def\sdist{{\rm sdist}}
\def\diam{{\rm diam}}
\def\spt{{\rm spt}}
\def\argmin{{\rm argmin}}

\title[A two-point function approach to connectedness of drops]{A two-point function approach to   connectedness of  drops in convex potentials}

\author[G. De Philippis]{Guido De Philippis}
\address{G.D.P.: SISSA, Via Bonomea 265, 34136 Trieste, Italy}
\email{guido.dephilippis@sissa.it}

\author[M. Goldman]{Michael Goldman}
\address{M.G.: Universit\'e Paris-Diderot, Sorbonne Paris-Cit\'e, Sorbonne Universit\'e,  CNRS,  Laboratoire Jacques-Louis Lions, LJLL, F-75013 Paris}
\email{goldman@math.univ-paris-diderot.fr}

\subjclass[2010]{49Q10 (58J05)}
\keywords{Two-point function, Second variation, Convex potential}

\begin{document}
\begin{abstract} 
We establish connectedness of volume constrained minimisers of energies involving surface tensions  and convex potentials.  By a previous result of McCann, this implies that  minimisers are convex in dimension two.
This positively answers   an old question of  Almgren.
We also prove convexity of minimisers when the volume constraint is dropped. Our proof is based on  the introduction of a new ``two-point function'' which measures the lack of convexity and which gives rise to a negative second variation of the energy.
\end{abstract}

\maketitle

\section{Introduction}
\subsection{Overview} Crystals and drops  subject to the action of an external potential are usually described by the following free energy (see \cite{hoffman})
\begin{equation}\label{e:en}
\mathcal F(E)=\int_{\partial ^*E} \Phi(\nu_{\partial^* E})+\int_E g,
\end{equation}
where the first integral is taken with respect to the $(d-1)-$dimensional Hausdorff measure and the second one with respect to the Lebesgue measure. 
Here \(E\subset \R^d\) is a set of finite perimeter representing  the volume occupied by the drop, \(\partial^*E\) denotes its  reduced boundary and  \(\nu_{\partial^*E}\)  its  outer normal (see  \cite{Maggi12} for more details). The potential   \(g:\R^d\to \R\) accounts for the external forces and
\(\Phi: \R^d\to \R\) is a  \emph{one-homogeneous} and {\em convex} function  which describes the (typically anisotropic) surface tension. Note that when the surface tension is isotropic, i.e. \(\Phi(\nu)=|\nu|\), the surface term reduces to the classical  De Giorgi perimeter, \(P(E)\).

It is commonly assumed that  minimisers of \eqref{e:en} under  a  volume constraint give a good description of the equilibrium shapes of drops. We thus  consider the following variational problem: 
\begin{equation}\label{e:min}\tag{\(P_V\)}
\min_{|E|=V} \mathcal F(E)
\end{equation}
and  its unconstrained counterpart:
\begin{equation}\label{e:minun}\tag{\(P\)}
\min_E \mathcal F(E).
\end{equation}

Let us also note that  the above variational problems naturally appear  in various other contexts such as:

\begin{itemize}
\item[-] In the limiting case in which \(g\) takes only the values \(0\) and \(+\infty\), \eqref{e:min} reduces to the isoperimetric problem in the domain \(\Omega=\{g<+\infty\}\) (see \cite{AlterCasellesChambolle05});
\item[-] The minimisation problem \eqref{e:minun} appears as one step of the Almgren, Taylor and Wang approximation of the   mean curvature flow (cf. \cite{AlmgrenTaylorWang93}).  There, \(g=\sdist(x,E_{k-1})\) is the signed distance to the \((k-1)\)-th step of the scheme, see also  \cite{LuckhausSturzenhecker95}.
\end{itemize} 

A natural question is to understand how properties of the surface tension \(\Phi\)  and of the potential \(g\) influence the shape of \(E\). In this paper we investigate the following question which is attributed to Almgren (see \cite{McCann98}):

\medskip
\noindent

{\bf Question:} {\em Let \(E\) be a minimiser of \eqref{e:min} and let us assume that \(g\) is convex. Is it true that \(E\) is convex?}

\medskip

Note that, if the  answer to this question is positive, it can only result from  a delicate interaction between the surface and the volume terms in \eqref{e:en}. Indeed, it is known that the answer is positive both for very small and very large volumes but
for totally different reasons.
On  the one hand, in the regime  \(V\ll 1\) and assuming that \(\Phi\) and \(g\) are sufficiently smooth, Figalli and Maggi   showed  in \cite{FigalliMaggi11} that  \(E\) is a smooth perturbation of the Wulff shape associated to \(\Phi\) (the ball if \(\Phi(\nu)=|\nu|\))
and is thus convex. On the other hand, in the regime \(V\gg1\), the set \(E\) should resemble the level set of \(g\) having  volume \(V\) and therefore must be  convex as well, see \cite{CasellesChambolle06} and  Remark \ref{rmk:largevol} below.

\medskip 

In the regime of intermediate mass   \(V\sim 1\), no term in the functional is predominant and  very little is known about the minimisers.  To the best of our knowledge, the only available result in this general  setting is due to McCann  (see \cite{McCann98}). Building on an unpublished paper of Okikiolu, he proved that   when \(d=2\), every connected component of \(E\) must be 
convex. Moreover, using ideas from optimal transport theory he proved that every such connected component is    uniquely minimising  \eqref{e:min}  for its own volume in the class of convex sets.

\subsection{Main results} The main result of this paper asserts that minimisers of \eqref{e:min} are always connected. More precisely we have the following result which is new even in the case of  isotropic surface tensions, \(\Phi(\nu)=|\nu|\).

\begin{theorem}\label{thm:main}
Let \(E\) be a minimiser of \eqref{e:min} and assume that \(g\in C^{1,\alpha}(\R^d)\) is a  convex and coercive function and that  \(\Phi\in C^{3,\alpha}(\R^d\setminus \{0\})\) is uniformly elliptic, i.e. 
\begin{equation}\label{e:unifellipt}
 \langle D^2 \Phi(\nu) \xi,\xi\rangle \ge (\xi- \langle \nu, \xi\rangle)^2   \qquad  \forall \ |\xi|=| \nu|=1.
\end{equation}
 Assume moreover then  $g$ is strictly  convex, then, \(\partial E\) is connected, in particular   \(E\)  is  indecomponible\footnote{Recall that a set of finite perimeter \(E\) is said to be indecomponible if for every partition \(E=E_1\cup E_2\) with \(|E_1\cap E_2|=0\) and \(P(E)=P(E_1)+P(E_2)\) then either \(|E_1|=0\) or \(|E_2|=0\). }  (and thus connected). In case \(d=2\), the strict convexity assumption is not needed.
\end{theorem}

From the result of  McCann (see \cite{McCann98}) this  immediately positively answers Almgren's question in dimension \(2\).

\begin{cor}\label{cor:2d}
Assume \(d=2\) and let \(E\subset \R^2\) be a minimiser of \eqref{e:min} with  \(g\in C^{1,\alpha}(\R^2)\) a convex and coercive function  and  \(\Phi\in C^{3,\alpha}(\R^2\setminus \{0\})\)  uniformly elliptic. Then  \(E\)  is convex and unique.
\end{cor}
\begin{remark}
 By approximation of $g$ and $\Phi$, this implies that if $d=2$, for every convex and coercive $g$ 
 and every convex anisotropy $\Phi$, for every volume $V$ there {\em exists} a minimizer of \eqref{e:min} which is convex. However, we can not rule out the existence of other, possibly disconnected, minimizers.
\end{remark}

For the unconstrained problem \eqref{e:minun} we are actually able to show convexity of minimisers.

\begin{theorem}\label{thm:main2}
Let \(E\) be a  minimiser of \eqref{e:minun} and assume that \(g\in C^{1,\alpha}(\R^d)\) is a convex and coercive function and that  \(\Phi\in C^{3,\alpha}(\R^d\setminus \{0\})\) satisfies   \eqref{e:unifellipt}, then  \(E\)  is convex. 
\end{theorem}
\begin{remark}
 In Theorem \ref{thm:main} we must require strict convexity of $g$ in the case $d>2$ since in that case we are not able to guarantee the existence of a stable connected component of $\partial E$ intersecting\footnote{Here and in the sequel \(\co(F)\) denotes the \emph{closed convex envelope} of a set \(F\).}  $\partial \co(E)$.

However, we  believe  that our results actually hold true for more general surface tensions $\Phi$ (for instance crystalline ones) and possibly non-smooth  convex potentials $g$.
\end{remark}

\begin{remark}\label{rmk:largevol}
As mentioned above, convexity of solutions of \eqref{e:min} is known for large volumes. Indeed, even though the precise statement  is not written explicitly anywhere, it follows from \cite[Theorem 3]{CasellesChambolle06} arguing exactly as in \cite[Lemma 5.1]{CasellesChambolle06} (see also  \cite[Theorem 5.6]{ChambolleGoldmanNovaga13} and  \cite[Theorem 11]{AlterCasellesChambolle05}).
The proof is based on a calibration technique to show that for large volumes, solutions of \eqref{e:min} are given by sublevel-sets of the (unique) local minimizer $u$ of 
\[
 \int |Du|+ \frac{1}{2} \int (u-g)^2.
\]
This is coupled with the Alvarez-Lasry-Lions convexity result \cite{AlvarezLasryLions97} (see also \cite{Korevaar}) for solutions of elliptic PDEs to obtain the convexity of $u$. One of the outcomes 
of the proof is the existence of  \(\bar V>0\) such that, for all \(V\ge \bar V\) there exists \(\lambda_V\in \R\) for which 
\[
\textrm{\(E\) minimises \eqref{e:min}}\quad\iff\quad E\in \argmin\Big\{ \int_{\partial ^*F} \Phi(\nu_{\partial^* F})+\int_F (g-\lambda_{V})\Big\}.
\]
Taking this  equivalence for granted, convexity of solutions of \eqref{e:min} for \(V\ge \bar V\) could also be obtained  by applying Theorem \ref{thm:main2} with \(g\) replaced by \(g-\lambda_{V}\). 
\end{remark}


%

\subsection{Idea of proof and structure of the paper} The  proofs of our main  results are based on the \emph{two-point function technique}. This technique, quite standard in the context of viscosity solutions \cite{CrandallIshiiLions92} (where it goes by the name of doubling of variables trick), 
has been introduced in the realm of geometric problems by Andrews in \cite{Andrews12}  to show preservation of an interior ball condition   along  the mean curvature flow. Brendle in \cite{Brendle13} improved   these ideas to  the ``elliptic'' setting  to  prove Lawson's conjecture about minimal torii in \(\mathbb S^3\). 
Several applications of these ideas can be find in the nice  survey paper of Andrews, \cite{Andrews15}.

 In the   examples mentioned above, the two-point function which is  used is a ``non-infinitesimal'' version of the norm of the  second fundamental form of \(\partial E\), namely
 \begin{equation*}
 I_{\partial E}(x)= 2\sup_{y\in \partial E} \frac{\langle \nu_{\partial E}(x), y-x\rangle}{|x-y|^2}.
 \end{equation*}
Assuming  for simplicity that \(\Phi(\nu)=|\nu|\) and  \(g=0\), this  can be  shown to be a subsolution of   Simons equation\footnote{We denote by $A_{\partial E}$ the second fundamental form of $\partial E$ and by $H_{\partial E}=\tr(A_{\partial E})$ its mean curvature.}:
\begin{equation*}
(\Delta+|A_{\partial E}|^2)I_{\partial E}\ge c(d)\frac{|\nabla I_{\partial E} |^2}{I_{\partial E}}-H_{\partial E} I_{\partial E}^2
\end{equation*}
or  of its parabolic variant. Then, suitable applications of the maximum principle allow to conclude.
  
  Here we use a different function which measures the ``non-convexity'' of \(\partial E\):
  \begin{equation}\label{e:Sintro}
  S_{\partial E}(x)=\sup_{y\in \partial E}\langle \nu_{\partial E}(x), y-x\rangle.
  \end{equation}
 See  also \cite{Weinkove18} where a similar function is used in a different context. This turns out to be a  positive subsolution of the Jacobi equation:
 \begin{equation*}
 (\Delta+|A_{\partial E}|^2)S_{\partial E}\ge0
 \end{equation*}
  and thus induces a negative second variation of the energy. In the case of the unconstrained minimisation problem  \eqref{e:minun}  this is enough to conclude since  
  every minimiser must have positive second variation. In  the case of the volume  constrained minimisation problem \eqref{e:min}, $S_{\partial E}$ is no longer  an admissible variation
  since  it does not respect the volume constraint  and hence we are not able to conclude that  \(E\) is convex. Nevertheless if \(\partial E\) is  disconnected  at least one of its connected components must  be stable with respect to all possible variations i.e. non necessarily volume preserving, and this gives the desired contradiction. 
 
 Let us  remark that in contrast to \(I_{\partial E}\), our function \(S_{\partial E}\) does not have an  ``infinitesimal'' version. However, while \(I_{\partial E}\) is a subsolution of Simons equation, ours is directly a subsolution of the Jacobi equation which is easier to use  in combination with the stability inequality. 
 Let us also note that the doubling of variable trick  allows to get rid of a term with an  unfavorable sign, see Lemma \ref{lm:key} below\footnote{This is related to the following simple observation: if a symmetric block matrix is positive,
 \[
 \begin{pmatrix}
 A & B
 \\
 B&-C 
 \end{pmatrix}
 \ge 
 \begin{pmatrix}
 0 & 0
 \\
 0&0 
 \end{pmatrix},
 \]
 then not only \(\tr (A)\ge \tr(C)\) but also \(\tr(A)\ge\tr(C) \pm 2\tr(B)\).}.

  We point out  that it should be clear from the above sketch of the argument that our results actually apply to (sufficiently smooth) stable critical points. Indeed, minimality is only used to overcome some regularity issues, see Theorem \ref{thm:stable} and  Remark \ref{rmk:stability} below. 

 Finally we would like to underline the fact  that the combination of the two-point function technique with the assumption of positive second variation is quite natural. In fact, on a general stationary point the Jacobi operator does not satisfy the maximum principle. It is precisely the stability condition which guarantees its validity.  We hope that this simple observation can be useful also in other contexts. 
 
 Let us also mention the works  \cite{ SternZum2, SternZum1} where    Sternberg and Zumbrun  used (in a quite different way) the  stability inequality to prove connectedness of minimizers of  the relative isoperimetric problem inside  (strictly) convex sets.

  \medskip 
  The paper is structured as follows: in Section \ref{sec:pre} we recall some well-known properties of minimisers of \eqref{e:min} and \eqref{e:minun}, together with some basic facts of differential geometry. In Section \ref{sec:proof} we prove Theorem \ref{thm:main} and \ref{thm:main2}. Eventually in the Appendix \ref{sec:ap} we prove a simple approximation lemma.

 \subsection*{Acknowledgements}
This work has been conceived while the first author was a FSMP visiting professor  at the Laboratoire LJLL in  Paris. The support of the FSMP and the nice working environment of the Laboratoire LJLL  are greatly acknowledged.
G.~D.~P.\ is supported by the MIUR SIR-grant ``Geometric Variational Problems" (RBSI14RVEZ).

\section{Notation and preliminaries}\label{sec:pre}

In this section we collect a few results that will be used in the sequel and we fix some notation.

\subsection{Differential geometry}Most of the concepts and relations introduced here may be find in \cite{Carmo92} (see in particular \cite[Chapter 2.3]{Carmo92} and \cite[Chapter 6.2]{Carmo92}).  We let $(e_1,\cdots, e_d)$ be the canonical basis of $\R^d$. Given  a symmetric \(k\)-linear form \(B\), \(B[v_1,\dots, v_k]\) denotes its action on the \(k\)-tuple of vectors \(v_1,\dots, v_k\).  From now on we will use the symbol \(D\) for the flat connection on \(\R^d\), i.e. the classical componentwise derivative. In coordinates, this means that if $X=\sum_{i=1}^d X_i e_i$,  then
$ D_X Y=\sum_{i=1}^d X_i \partial_i Y$.

For a \(C^2\) regular \((d-1)\)-manifold \(M\) oriented by its  normal \(\nu_M\) and a vector field $X$, we will use \(\nabla_X\) to denote the covariant derivative on \(M\):
\[
\nabla _X Y=\boldsymbol{p}_{T_x M} D_XY
\]
where \(\boldsymbol{p}_{T_x M}\) is the orthogonal projection onto \(T_x M\). Let us recall that for $X$, $Y$, $Z$ vector fields, we have 
\begin{equation}\label{compatibility}
   D_X  \langle Y,Z\rangle= \langle D_X Y,Z\rangle +\langle Y, D_X Z\rangle,
\end{equation}
which can be easily checked using coordinates (see also \cite[Corollary 3.3]{Carmo92}). In particular, if $X$ and $Y$ are tangent then $\langle D_X Y,\nu_M\rangle=- \langle Y, \nabla_X \nu_M\rangle$. We may thus define the second fundamental form $A_M$ by its action of tangent vector fields as 
\begin{equation}\label{e:sff}
A_M(x)[X(x), Y(x)]= -\langle D_X Y(x),\nu_M(x)\rangle=\langle \nabla_X \nu_M(x),Y(x)\rangle=\langle \nabla_Y \nu_M(x),X(x)\rangle,
\end{equation}
where the last equality follows from the symmetry of $A_M$ (see \cite[Proposition 2.1]{Carmo92}). Note that with our convention \(A_M\) is positive if \(M\) is the boundary of a convex set oriented by its exterior normal.

   Given two tangent vector fields \(X\) and \(Y\) we have by \eqref{e:sff}
\begin{equation}\label{e:split}
D_X Y=\nabla_X Y+\langle D_X Y,\nu_M\rangle \nu_M=\nabla_X Y-A_M[ X, Y] \nu_M.
\end{equation}
For a \(C^1\) function \(f\) defined in a neighborhood of \(M\) we define  its tangential gradient \(\nabla f(x)\in  T_xM\)  as 
\[
\nabla f =\boldsymbol{p}_{T_x M} Df.
\] 
If  \(X\) is a field of tangent vectors,  we  set \(\nabla_X f=\langle \nabla f, X \rangle =\langle D f, X \rangle\). The  tangential Hessian \(\nabla^2 f (x): T_{x}M\times T_x M\to \R\) is given by its action on   tangential vector fields \(X\) and \(Y\) defined in a neighborhood of \(x\) as 
\begin{equation}\label{e:hessian-1}
\begin{split}
\nabla^2 f(x) [X(x),Y(x)]&=\langle \nabla_{X}(\nabla f),Y\rangle(x)\stackrel{\eqref{compatibility}}{=}\nabla_{X} (\nabla_{Y} f)(x)-\langle \nabla  f(x), D_X Y(x)\rangle 
\\
 &=D^2 f(x)[X(x),Y(x)]-A_M[X,Y](x)\langle Df(x),\nu_M(x)\rangle
\end{split}
\end{equation}
where  we have used that $\nabla_{X} (\nabla_{Y} f)= D^2 f[X,Y]+\langle D f,D_X Y\rangle$ and that 	by \eqref{e:split}, 
\[\langle D f, D_X Y\rangle= \langle \nabla  f, D_X Y\rangle- A_M[X,Y]\langle Df ,\nu_M\rangle.\]
Note  that the definition  of $\nabla^2 f$ depends only on the values of \(f\) on \(M\) and of \(X\) and \(Y\) at \(x\). The Laplace-Beltrami operator of \(f\) is then given by  
\[
\Delta_M f =\tr \nabla^2 f=\Div_{M} (\nabla f)
\]
where for a (not necessarily tangent) vector field \(Y\), the tangential divergence \(\Div_M \, Y\) is defined as
\[
\Div_M \, Y(x)= \sum_{i=i}^{d-1} \langle D_{\tau_i} Y(x), \tau_i(x)\rangle
\]
for a local tangent  orthonormal  frame \(\{\tau_i\}\) around \(x\). Observe  that this definition does not depend on the chosen frame.

If   \(\Phi\in C^2(\R^d\setminus\{0\})\) is one-homogeneous,   the {\em anisotropic mean curvature} \(H^\Phi_M\) of  an oriented \((d-1)\) manifold  \(M\) is defined  as
\begin{equation}\label{def:HPhi}
H^\Phi_M=\Div_{M} (D \Phi(\nu_M))=\tr ((D^2 \Phi (\nu_M) A_M). 
\end{equation}
When \(\Phi(\nu)=|\nu|\), this reduces to the classical mean curvature. Note that this definition is well posed since, by one-homogeneity, \(D^2\Phi(\nu_M(x))\nu_M(x)=0\) 
and thus \(D^2\Phi(\nu_M(x)): T_x M\to T_x M\). For the same reason the term  \(\tr(D^2\Phi(\nu_{\partial E}) A_{\partial E}^2)\) appearing in \eqref{e:stab} below is well defined.

\subsection{Properties of minimisers}
The following theorem summarises well-known properties of minimisers of \eqref{e:minun} and \eqref{e:min}.

\begin{theorem}\label{thm:promin}
Let us assume  that \(g\in C^{1,\alpha}(\R^d)\) satisfies 
\[
\lim_{|x|\to +\infty} g(x)=+\infty
\] 
 and that  \(\Phi\in C^{3,\alpha}(\R^d\setminus \{0\})\) is uniformly elliptic (recall \eqref{e:unifellipt}).  Then, for every \(V\in (0,+\infty)\) there exists a minimiser  of \eqref{e:min} (resp. \eqref{e:minun}).  Moreover any minimiser \(E\) of \eqref{e:min} (resp. \eqref{e:minun}) satisfies:
 \begin{itemize}
 \item[(i)] \(E\) is equivalent to an open bounded set.
 \item[(ii)]  Let \(\Sigma\) be the singular set of \(\partial E\), i.e.
 \[
 \Sigma=\big\{x\in \partial E \ : \  \partial E \textrm{ does not have a tangent plane at \(x\)}\big\}.
 \]
 Then, \(\Sigma\) is closed,  \(\HH^{d-3}(\Sigma)=0\) and  for all  \(x\in \partial E\setminus \Sigma\) there exists a neighbourhood \(U_x\) such that   \(E\cap U_x\) can be locally written as the epi-graph of a \(C^3\) function. In particular,  \(\partial E\setminus \Sigma\)   is a (relatively open) \(C^3\) manifold oriented by \(\nu_E\).
 \item[(iii)]There exists a constant \(\mu\) such that 
\begin{equation}\label{e:stat}
H_{\partial E}^ \Phi+g=\mu\qquad \textrm{for all \(x\in \partial E\setminus \Sigma \)}\qquad \textrm{(resp \(H_{\partial E}^ \Phi+g=0\)),}
\end{equation}
and
\begin{multline}\label{e:stab}
\int_ {\partial E\setminus \Sigma } \langle D^2 \Phi(\nu_{\partial E}) \nabla \varphi , \nabla \varphi  \rangle -\tr(D^2\Phi(\nu_{\partial E}	) A_{\partial E}^2) \varphi^2+ D_\nu g \,\varphi^2\ge 0
\\
\textrm{for all \(\varphi\in C^1_c(\partial E \setminus \Sigma) \) such that \(\int_{\partial E} \varphi=0\) (resp. for all \(\varphi\in C^1_c(\partial E \setminus \Sigma) \)).}
\end{multline}
\item[(iv)] Assume moreover that \(g\) has convex level sets, then 
\begin{equation}\label{e:minextern}
\int_{\partial E} \Phi(\nu_{\partial E})\le  \int_{\partial^* F} \Phi(\nu_{\partial^* F})\qquad\textrm{for all \(F\supset E\)}.
\end{equation}
\end{itemize}
\end{theorem} 

Note that, if \(\Phi(\nu)=|\nu|\), then  the measure estimate in (ii) can be upgraded to \(\HH^{d-8+\eps}(\Sigma)=0\) for all \(\eps>0\).

\begin{proof}
The existence of a minimiser and the  boundedness of every minimiser is standard, see for instance \cite{FigalliMaggi11,De-PhilippisMaggi15} and \cite[Section 4.2]{FigalliMaggi11} for what concerns boundedness. Moreover, one can easily show that   there exist constants  \(\Lambda, r_0>\) depending only on the data such that any minimiser of \eqref{e:min} (resp. \eqref{e:minun})  is a \((\Lambda, r_0)\) minimiser of the energy \(F\mapsto \int_{\partial F} \Phi (\nu_{\partial F})\), namely 
\begin{equation}\label{e:Lambdamin}
\int_{\partial^* E} \Phi (\nu_{\partial^* E})\le \int_{\partial ^*F} \Phi (\nu_{\partial^* F})+\Lambda|E\Delta F|\qquad \textrm{for all \(F\) such that \(\diam(E\Delta F)\le r_0\)},
\end{equation}
see for instance \cite[Example 21.2]{Maggi12}. Claim (i) and (ii) follow, see \cite{Almgren76, Bombieri82,SchoenSimonAlmgren77} or \cite{De-PhilippisMaggi15,De-PhilippisMaggi17} where the theory is explained in the context of sets of finite perimeter.
Equations \eqref{e:stat} and \eqref{e:stab} come  from the first variation and second variation of \(\mathcal F\), see for instance \cite[Appendix A]{FigalliMaggi11} or the appendix of \cite{De-PhilippisMaggi17}.  The outer minimising property (iv) is established  in \cite[Appendix B]{FigalliMaggi11}, see also \cite{McCann98}.

\end{proof}


The  following is a simple criterion for regularity, its proof is classical and we  sketch it here for the  reader's convenience.

\begin{lemma}\label{lm:reg}
Let \(\Phi\) and \(g\) be as in Theorem \ref{thm:promin} and let $E$ be a minimizer of \eqref{e:min} or \eqref{e:minun}. Assume that \(\bar x\in \partial E\) and that  there exist \(r>0\) and a set \(G\) with \(C^1\) boundary such that 
\[
E\cap B(\bar x,r)\subset G\cap B(\bar x,r)\qquad \bar x \in \partial E\cap \partial G.
\] 
Then \(\bar x\) is a regular point, namely \(\bar x \in \partial E \setminus \Sigma\). 
\end{lemma}

\begin{proof}
By the classical  \(\varepsilon\)-regularity criteria, see for instance \cite[Section 3]{De-PhilippisMaggi15} or \cite{Bombieri82}, it is enough to show that there exists a sequence \(\varrho_k\downarrow 0\) and a half-space \(H\)  such that 
\[
E_k=\frac{E-\bar x}{\varrho_k}\to H.
\]
By classical density estimates the sequence \(\{E_k\}\) is relatively compact in the \(L^1_{\rm loc}\) topology and, by assumption,  any of its limit points \(E_\infty\) satisfies
\[
0\in \partial E_\infty\qquad E_\infty\subset H_{\nu_{\partial G}}=\{y \ : \ \nu_{\partial G}(\bar x)\cdot y\le 0\}.
\]
Moreover, using \eqref{e:Lambdamin} we get that \(E_\infty\) is a (unconstrained) minimiser of \(F\mapsto \int_{\partial F} \Phi (\nu_{\partial F})\). The maximum principle, see \cite{SolomonWhite89} or \cite[Lemma 2.13]{De-PhilippisMaggi15}, then forces \(E_\infty= H_{\nu_{\partial G}}\) and this concludes the proof.
\end{proof}

\section{Proof of Theorems \ref{thm:main} and \ref{thm:main2} }\label{sec:proof}

In this section we will prove Theorem \ref{thm:main} and \ref{thm:main2}. The main calculation consists in establishing  for a generic set $E$, a  differential inequality for the function $S_{\partial E}$ introduced in \eqref{e:Sintro}. 
Since these computations  do not depend  on the nature of the problem and since they can be of independent interest, we isolated them in the next subsection.

\subsection{The key computation} Let \(E\subset \R^d \) be a bounded  open set such that  its boundary   can be split as 
\[
\partial E= R_{\partial  E}\cup \Sigma_{ \partial E} 
\]
where \(R_{\partial E}\) is a \(C^3\) manifold oriented by \(\nu_{\partial E}\) and \(\Sigma_{\partial E}\) is a closed singular set with empty relative interior. Given \(x\in R_{\partial E}\) and \(y\in \partial E\), let us define 
\[
S_{\partial E}(x,y)=\langle \nu_{\partial E}(x),y-x\rangle
\]
and 
\begin{equation}\label{e:S}
S_{\partial E}(x)=\max_{y\in \partial E} S_{\partial E}(x,y)\ge 0,
\end{equation}
whee the maximum is attained since $\partial E$ is compact. Note that \(S_{\partial E}(x) \) is  defined only for \(x\in R_{\partial E}\) and it is  the supremum of a family of locally  uniformly \(C^2\) functions. In particular it is  locally Lipschitz and locally semi-convex.\footnote{This means that for any local chart \(\phi: M \to \R^{d-1}\),  the function \(S_{\partial E}(\phi^{-1}(x))+C|x|^2\) is convex for a suitable constant \(C\). Note that this does not depend on the choice of the chart.}
The following lemma  is elementary. 
\begin{lemma}\label{lm:conv}
Let \(x\in R_{\partial E}\), then  \(S_{\partial E}(x)=0\) if and only if \(x\in \partial \co(E)\cap R_{\partial E}\).
\end{lemma}

\begin{proof}
\(S_{\partial E}(x)=0\) if and only if the hyperplane \(H=\{y: \langle\nu_{\partial E}(x), y-x\rangle\le 0\}\) is a supporting plane to \( E\) at \(x\), which is equivalent to \(x\in \partial \co(E)\cap R_{\partial E}\). 
\end{proof}


We now  define for \(\varphi \in C^2 (R_{\partial E})\), the Jacobi operator
\begin{equation}\label{def:LPhi}
\begin{split}
L_\Phi \varphi&=\Div_{\partial E} (D^2\Phi(\nu_{\partial E}) \nabla \varphi)+\tr(D^2 \Phi(\nu_{\partial E}) A_{\partial E}^2) \varphi.
\end{split}
\end{equation}
Notice that by integration by parts, if $\varphi\in C^2_c (R_{\partial E})$
\begin{equation}\label{e:ipp}
 \int_{R_{\partial E}} (-L_\Phi \varphi)\varphi+ D_\nu g \,\varphi^2= \int_{R_{\partial E}} \langle D^2 \Phi(\nu_{\partial E}) \nabla \varphi , \nabla \varphi  \rangle -\tr(D^2\Phi(\nu_{\partial E}) A_{\partial E}^2) \varphi^2+ D_\nu g \,\varphi^2.
\end{equation}
By a slight abuse of notation, we will identify the left-hand side and the right-hand side  of \eqref{e:ipp} even when $\varphi\in W^{1,2}(R_{\partial E})$.\\
The following is the key computation. Note that, thanks to Lemma \ref{lm:reg}, in the context of Theorem \ref{thm:main} the condition \(\bar y\in R_{\partial E}\) will be always satisfied (see Lemma \ref{lm:keyupgrade} below).
\begin{lemma}\label{lm:key} 
Let \(E\) be as above and let \(\bar x\in R_{\partial E}\). Assume that 
\[
S_{\partial E}(\bar x)=S_{\partial E}(\bar x, \bar y)
\]
with  \(\bar y\in R_{\partial E}\). Then, recalling the definition \eqref{def:HPhi} of   \(H^{\Phi}_{\partial E}\),
\begin{equation*}
L_\Phi S_{\partial E}(\bar x)\ge H^{\Phi}_{\partial E}(\bar x)-H^{\Phi}_{\partial E}(\bar y)+\langle \nabla H^{\Phi}_{\partial E}(\bar x), \bar y-\bar x\rangle
\end{equation*}
in the viscosity sense, meaning that for all \(\varphi \in C^2(R_{\partial E})\) such that 
\[
\varphi(x)-S_{\partial E}(x)\ge \varphi(\bar x)-S_{\partial E}(\bar x)=0,
\]
then 
\[
L_\Phi \varphi (\bar x)\ge H^{\Phi}_{\partial E}(\bar x)-H^{\Phi}_{\partial E}(\bar y)+\langle \nabla H^{\Phi}_{\partial E}(\bar x), \bar y-\bar x\rangle.
\]
\end{lemma}

\begin{proof}
Since the set  $E$ is fixed, for notational simplicity we drop the dependence on $E$ of the various quantities.  
Let \(\varphi\) be as in the statement of the lemma and note that the function
\[
G(x,y)=\varphi(x)-S(x,y) 
\]
achieves its minimum at \((\bar x, \bar y)\). Moreover, by assumption, it is \(C^2\) in a neighborhood of \((\bar x, \bar y)\).

\medskip
\noindent
{\em Step 1: Optimality conditions}:  In order to prove the theorem we are going to exploit the first and second order minimality conditions for \(G\) at the point \((\bar x, \bar y)\). To this end let us  choose two  local orthonormal frames \(\{\tau_i^x\}_{i=1}^{d-1}\) and \(\{\tau_i^y\}_{i=1}^{d-1}\)  around \(\bar x\) and \(\bar y\) respectively such that:
 \begin{align}
\langle \tau_i^x, \tau^x_j \rangle (\bar x) =\delta_{ij}, &\qquad \langle \tau^y_i, \tau^y_j \rangle(\bar y) =\delta_{ij}, \label{e:on1}
\\
(\nabla_{\tau^x_{i}}\tau^x_j) (\bar x)=0, &\qquad (\nabla_{\tau^y_{i}}\tau^y_j) (\bar y)=0.\label{e:osc-1}
\end{align}
Here \(\delta_{ij}\) is the Kronecker delta.  Letting  \(A_{ij}(x)= A(x)[\tau^x_i,\tau_j^x]\) we have from \eqref{e:sff} 
\begin{equation}\label{e:osc2}
\langle D_{\tau^x_{i}}\nu  ,\tau^x_j\rangle(\bar x)=A_{ij}(\bar x) \qquad \textrm{and} \qquad \langle  D_{\tau^y_{i}}\nu ,\tau^y_j\rangle(\bar y)=A_{ij}(\bar y). 
\end{equation}
Combing this with \eqref{e:split} and \eqref{e:osc-1}   we then obtain
\begin{equation}\label{e:osc}
D_{\tau^x_{i}}\tau^x_j (\bar x)=-A_{ij}(\bar x)\nu(\bar x)  \qquad \textrm{and} \qquad D_{\tau^y_{i}}\tau^y_j(\bar y)=-A_{ij}(\bar y)\nu(\bar y). 
\end{equation}

To simplify our notation, for a function \(f:\partial E\times \partial E\subset \R^d\times \R^d\to \R\) we set 
\[
\nabla_{i}^x f=\langle \nabla f,\tau_{i}^x\rangle \qquad \textrm{and} \qquad  \nabla_{i}^y f=\langle \nabla f,\tau_{i}^y\rangle, 
\]
where by a slight abuse of notation, we identify $\tau_i^x$ with $(\tau_i^x,0)\in T_x \partial E\times T_x \partial E$ and $\tau_i^y$ with the vector $(0,\tau_i^y)$. Similarly, we set 
\[
\nabla_{i}^x\nabla_j^x f= \nabla^2 f[\tau_{i}^x,\tau^x_{j}], \qquad \nabla_{i}^y\nabla_j^y f= \nabla^2 f[\tau_{i}^y,\tau^y_{j}],   \quad \textrm{and } \quad  \nabla_{i}^x\nabla_j^y f= \nabla^2 f[\tau_{i}^x,\tau^y_{j}]. 
\]
Note that thanks to \eqref{e:osc-1}, by \eqref{e:hessian-1} we have
\begin{equation}\label{e:hessian}
\nabla_{i}^x\nabla_j^x f(\bar x,\bar y)=\nabla_{i} ^x(\nabla_j^x f)(\bar x, \bar y) 
\end{equation}
and a similar expression for \(\nabla_{i}^y\nabla_j^y  f(\bar x,\bar y)\). Moreover, the trivial  relations
\[
\nabla_{\tau_i^x}\tau_{j}^y=0=\nabla_{\tau_j^y}\tau_{i}^x,
\]
also give
\begin{equation}\label{e:hessian2}
\nabla_{i}^x\nabla_j^y f=\nabla_{j}^y\nabla_i^x f=\nabla_{i}^x(\nabla_j^y f)=\nabla_{j} ^y(\nabla_i^x f). 
\end{equation}

We are now ready to compute the optimality conditions. To simplify our formulas, from now on we  we will always exploit Einstein summation convention over repeated indices. 

 Since \(G\) achieves its minimum at \((\bar x, \bar y)\) we have
\begin{align}
0&=\nabla_{i}^x G(\bar x,\bar y)=\nabla_{i}^x\varphi(\bar x)-\nabla_{i}^x S(\bar x, \bar y) \label{e:opt1},
\\
0&=\nabla_{i}^y G(\bar x,\bar y)=-\nabla_{i}^y  S(\bar x, \bar y). \label{e:opt2}
\end{align}
Using that \( \nabla_{\tau_i^x} x=\tau^x_i\) is perpendicular to \(\nu (x)\), we have by \eqref{e:osc2}
\begin{equation}\label{e:gradS}
\nabla_i^x S(x,y) =A_{ik}(x) \langle \tau^x_k,  y- x \rangle\qquad \textrm{and} \qquad  \nabla_{i}^y S(x,y) =\langle \nu(x), \tau^y_i\rangle.
\end{equation}
The first equation and \eqref{e:opt1} then  give 
\begin{equation}\label{e:gradphi}
\nabla_{i}^x \varphi(\bar x) =A_{ik}(\bar x) \langle \tau^x_k(\bar x), \bar y-\bar x \rangle. 
\end{equation}
Equation \eqref{e:opt2} instead gives that
\[
0=\langle \nu(\bar x), \tau_i^y(\bar y) \rangle,
\]
and thus \(\nu(\bar x)=\pm \nu (\bar y) \).  Actually since \(y\mapsto S(\bar x,y)\) attains its maximum at \(\bar y\), we must have  
\begin{equation}\label{e:normpar}
\nu(\bar x)=\nu (\bar y).
\end{equation}
In particular \(T_{\bar x} \partial E= T_{\bar y} \partial E\) and  the frames may be chosen such that 
\begin{equation}\label{e:on2}
\langle \tau^x_i(\bar x), \tau^y_j (\bar y)\rangle =\delta_{ij}.
\end{equation}
We now compute the Hessian of \(G\). 
 Recalling \eqref{e:hessian},
\begin{equation}\label{e:xx}
\begin{split}
\nabla_{i}^x\nabla_{j}^x G(\bar x, \bar y) &=\nabla_i^x\nabla_j^x \varphi(\bar x) -\nabla_i^x\nabla_j^x S(\bar x,\bar y)\\
&\stackrel{\eqref{e:gradS}\&\eqref{e:osc2}}{=}\nabla_{i}^x \nabla_{j}^x  \varphi(\bar x) -\nabla_{i}^x A_{jk}(\bar x) \langle \tau^x_k(\bar x), y-x \rangle \\
&\qquad \qquad  -A_{jk}(\bar x) \langle D_{\tau_i^x} \tau_k^x(\bar x), \bar y-\bar x\rangle + A_{jk}(\bar x)\langle \tau_k^x,\tau_i^x\rangle(\bar x)\\
&\stackrel{\eqref{e:osc}\&\eqref{e:on1}}{=}\nabla_{i}^x \nabla_{j}^x  \varphi(\bar x) -\nabla_{i}^x A_{jk}(\bar x) \langle \tau^x_k(\bar x), y-x \rangle \\
&\qquad \qquad +A_{jk}(\bar x)A_{ki}(\bar x)  \langle \nu(\bar x), \bar y-\bar x \rangle+A_{ij}(\bar x)
\\ 
&= \nabla_{i}^x \nabla_{j}^x  \varphi(\bar x) -\nabla_{i}^x A_{jk}(\bar x) \langle \tau^x_k(\bar x),\bar y-\bar x \rangle 
 +A_{jk}(\bar x)A_{ki}(\bar x) \varphi(\bar x)+A_{ij}(\bar x) .  
\end{split}
\end{equation}
where  in the last equality we have used  that \(\varphi(\bar x)=S(\bar x, \bar y)=\langle \nu(\bar x), \bar y-\bar x \rangle\).
Analogously, using  \eqref{e:hessian2} and \eqref{e:gradS} 
\begin{equation}\label{e:xy}
\nabla_{i}^x\nabla_{j}^y G(\bar x, \bar y) =- \langle D_{\tau_i^x} \nu(\bar x),\tau_j^y(\bar y)\rangle \stackrel{\eqref{e:osc2}}{=}-  A_{ik}(\bar x) \langle \tau_k^x(\bar x), \tau_j^y(\bar y)\rangle \stackrel{\eqref{e:on2}}{=}-A_{ij}(\bar x).
\end{equation}
As for the \(y\) derivatives we get, using  \eqref{e:gradS} and \eqref{e:osc},
\begin{equation}\label{e:yy}
\nabla_{i}^y \nabla_{j}^y G(\bar x, \bar y) =A_{ij}(\bar y) \langle \nu(\bar x), \nu (\bar y)\rangle\stackrel{\eqref{e:normpar}}{=} A_{ij}(\bar y).
\end{equation}

\medskip
\noindent
{\em Step 2: The case of the area functional}. To make the computations clearer,  we first treat the particular case in which \(\Phi=|\cdot |\). Note that in this setting (recall \eqref{def:LPhi})
\[
L_{\Phi}\varphi = \Delta_{\partial E}\,  \varphi+|A|^2 \varphi,
\]
and $H^\Phi=A_{ii}$ is the classical mean curvature. Since \((\bar x, \bar y)\) is a minimum point for \(G\), we have
\[
0\le  (\nabla_{i}^x+\nabla_{i}^y)(\nabla_{i}^x+\nabla_{i}^y) G (\bar x, \bar y)= \nabla_{i}^x\nabla_{i}^x G(\bar x, \bar y)+2\nabla_{i}^x\nabla_{i}^yG(\bar x, \bar y)+\nabla_{i}^y\nabla_{i}^y G(\bar x, \bar y).
\]
Using this,  \eqref{e:xx}, \eqref{e:xy} and \eqref{e:yy}, we get
\begin{equation}\label{e:quasi}
\begin{split}
\Delta_{\partial E}\,   \varphi(\bar x) &\ge  \nabla_{i}^x A_{ik}(\bar x) \langle \tau^x_k(\bar x), y-x \rangle -A_{ik}(\bar x) A_{ki}(\bar x) \varphi(\bar x) -A_{ii}(\bar x) +2 A_{ii}(\bar x)-A_{ii}(\bar y)\\
&= \nabla_{i}^x A_{ik}(\bar x) \langle \tau^x_k(\bar x), y-x \rangle -|A|^2 \varphi(\bar x)+H^{\Phi}(\bar x)-H^{\Phi}(\bar y).
\end{split}
\end{equation}
We now use Codazzi equations 
\begin{equation}\label{e:codazzi}
\nabla_{i}^x A_{jk}=\nabla_{k}^x A_{ij},
\end{equation}
to obtain that
\begin{equation*}
\nabla_{i}^x A_{ik}=\nabla_{k}^x A_{ii}=\nabla_{k}^x H^{\Phi} ,
\end{equation*}
which, combined with  \eqref{e:quasi} yields
\begin{equation*}
\Delta_{\partial E} \varphi(\bar x)+|A|^2\varphi(\bar x)\ge H^{\Phi}(\bar x)-H^{\Phi}(\bar y) +\langle \nabla H^{\Phi}(\bar x),\bar y-\bar x\rangle.
\end{equation*}
This  proves the lemma for \(\Phi=|\cdot|\)

\medskip
\noindent
{\em Step 3: The general case}. To deal with the general case we start by writing \(L_{\Phi}\) in non-divergence form. Let \(\{\tau^x_i\}\) be an orthonormal local frame  satisfying \eqref{e:on1}, \eqref{e:osc-1}, \eqref{e:osc2}, \eqref{e:osc} and let us set, for \(i,j,k\in\{1,\dots,(d-1)\}\),
\begin{equation*}
\Phi_{ij}(\nu)(x)=D^2\Phi (\nu(x))[\tau^x_{i}, \tau^x_{j}]\qquad \textrm{and }\qquad  \Phi_{ijk} (\nu)(x)=D^3\Phi(\nu(x))[\tau_i^x,\tau_{j}^x,\tau^x_{k}].
\end{equation*}
Note that, since \(D^2 \Phi(\nu)[\nu, \cdot]=0\) (by one-homogeneity) and since for all \(i\) and \(j\), \(D_{\tau_j^x}\tau_{i}^x(\bar x)\) is parallel to \(\nu(\bar x) \) (by \eqref{e:osc}) we  have that, at \(\bar x\),
\begin{equation}\label{e:3der}
\begin{split}
D_{\tau_k^x}\Phi_{ij}(\nu)&=D_{\tau_k^x} \big (D^2\Phi (\nu)[\tau^x_{i}, \tau^x_{j}]\big)
\\
&=D^3\Phi(\nu)[D_{\tau_{k}^x}\nu, \tau^x_{i}, \tau^x_{j}]+D^2\Phi(\nu)[D_{\tau_{k}^x}\tau^x_{i}, \tau^x_{j}]+D^2\Phi(\nu)[\tau^x_{i}, D_{\tau_{k}^x}\tau^x_{j}]
\\
&= \Phi_{ijl}(\nu) A_{lk},
\end{split}
\end{equation}
where we used  \eqref{e:osc2} and that by \eqref{compatibility}, $\langle D_{\tau_k^x}\nu,\nu\rangle=0$. For  \(\varphi\in C^2(R_{\partial E})\)  we then have at \(\bar x\)  
\[
\begin{split}
\Div_{\partial E} (D^2\Phi(\nu) \nabla \varphi)&= \big\langle D_{\tau^x_i} \big(D_{jk}^2\Phi(\nu)\nabla_{j}^x\varphi \,\tau^x_{k}\big),\tau^x_{i}\big\rangle 
\\
&\stackrel{\eqref{e:on1}\&\eqref{e:osc}}{=}  D_{\tau^x_i} \big(D_{ij}^2\Phi(\nu)\big) \nabla_{j}^x \varphi+\Phi_{ij}(\nu)  \nabla_i^x \big(\nabla_{j}^x\varphi \big)
\\
&\stackrel{\eqref{e:3der}}{=}\Phi_{ijl} A_{li}\nabla_{j}^x \varphi+\Phi_{ij}(\nu)\nabla_{i}^x\nabla_{j}^x \varphi.
\end{split}
\]
Hence, again at \(\bar x\), recalling   the definition of $L_\Phi$, \eqref{def:LPhi},
\begin{equation}\label{e:L}
\begin{split}
L_\Phi \varphi(\bar x)&=\Div_{\partial E} (D^2\Phi(\nu) \nabla \varphi)+\tr(D^2 \Phi(\nu) A^2) \varphi\\
&=\Phi_{ijl} A_{li}\nabla_j^x\varphi+\Phi_{ij}(\nu) \nabla_i^x\nabla_j^x\varphi+\Phi_{ij}(\nu) A_{ik}A_{kj} \varphi.
\end{split}
\end{equation}
Furthermore by \eqref{e:normpar} and \eqref{e:on2}, \(\Phi_{ij}(\nu)(\bar x)=\Phi_{ij}(\nu)(\bar y)\) and these are the components of a positive matrix. Hence, by minimality,
\[
0\le\Phi_{ij}(\nu)\nabla_i^x\nabla_{j}^x G(\bar x, \bar y)+2\Phi_{ij}(\nu)\nabla_i^x\nabla_{j}^y G(\bar x, \bar y)+\Phi_{ij}(\nu)\nabla_i^y  \nabla_{j}^y G(\bar x, \bar y).
\]
Using \eqref{e:xx}, \eqref{e:xy} and \eqref{e:yy} and  recalling the definition of $H^\Phi$, \eqref{def:HPhi},
\begin{equation}\label{e:mc}
H^\Phi=\Phi_{ij}(\nu)A_{ij}
\end{equation}
this implies
\begin{equation*}
\begin{split}
\Phi_{ij}(\nu)\nabla^x_i\nabla^x_j\varphi(\bar x)&\ge \Phi_{ij}(\nu) \nabla_{i}^xA_{jk}(\bar x) \langle \tau^x_k(\bar x), \bar y-\bar x \rangle 
\\
&-\Phi_{ij}(\nu) A_{jk}(\bar x)A_{ki} (\bar x)\varphi(\bar x)+H^\Phi(\bar x) -H^ \Phi(\bar y),
\end{split}
\end{equation*}
and thus by \eqref{e:L},
\begin{equation}\label{e:edai2}
 L_\Phi \varphi(\bar x)\ge \Phi_{ij}(\nu) \nabla_{i}^xA_{jk}(\bar x) \langle \tau^x_k(\bar x), \bar y-\bar x \rangle +\Phi_{ijl}(\nu) A_{li}(\bar x)\nabla_j^x \varphi(\bar x)+H^\Phi(\bar x) -H^ \Phi(\bar y).
\end{equation}

We now differentiate \eqref{e:mc} with respect to \(\tau_{k}^x\),   use \eqref{e:3der},  Codazzi equations \eqref{e:codazzi} and the symmetry of $\Phi_{ijl}$  to get at the point \(\bar x\), 
\[
\begin{split}
\nabla^x_{k} H^\Phi&=\Phi_{ij}(\nu)\nabla^x_{k}A_{ij}+\Phi_{ijl}(\nu)A_{lk}A_{ij}
\\
&=\Phi_{ij}(\nu)\nabla^x_{i}A_{jk}+\Phi_{ijl}(\nu)A_{jk}A_{li}.
\end{split}
\]
Multiplying this  equation by \(\langle \tau_{k}^x(\bar x),\bar y-\bar x\rangle\) and summing over \(k\) we get
\[
\begin{split}
\langle\nabla H^\Phi (\bar x),\bar y-\bar x\rangle&=\Phi_{ij}(\nu)\nabla_{i}^xA_{jk}(\bar x) \langle \tau^x_k(\bar x), \bar y-\bar x \rangle+\Phi_{ijl}(\nu)A_{jk}(\bar x) A_{li}(\bar x) \langle \tau^x_k(\bar x), \bar y-\bar x \rangle
\\
&\stackrel{\eqref{e:gradphi}}{=}\Phi_{ij}(\nu)\nabla_{i}^xA_{jk}(\bar x) \langle \tau^x_k(\bar x), \bar y-\bar x \rangle+ \Phi_{ijl} (\nu) A_{li}(\bar x) \nabla_{j}^x\varphi(\bar x).
\end{split}
\]
Plugging the above expression into \eqref{e:edai2}  we thus get
\[
L_\Phi \varphi (\bar x) \ge H^\Phi(\bar x) -H^ \Phi(\bar y) + \langle\nabla H^\Phi (\bar x),\bar y-\bar x\rangle,
\]
which concludes the proof.
\end{proof}

\subsection{Proof of Theorems \ref{thm:main} and \ref{thm:main2}} We now combine Lemma \ref{lm:key} with the convexity of $g$ to show that for minimizers of either \eqref{e:min} or \eqref{e:minun}, $S_{\partial E}$ gives a negative second variation. As already pointed out, 
this is the key observation for the proof of both Theorem \ref{thm:main} and Theorem \ref{thm:main2}. 

For $g\in C^{1}(\R^d)$, we say that  an increasing function $\omega$ is a modulus of convexity of $g$ if  
\begin{equation}\label{e:strictconv}
 g(y)-g(x)-\langle D g(x),y-x\rangle\ge \omega(|y-x|) \qquad \forall x, y \in \R^d. 
\end{equation}
Notice that of course, for every convex function zero is a modulus of convexity and if $g$ is  strictly  convex, it has a strictly positive modulus of continuity.

\begin{lemma}\label{lm:keyupgrade}
Let $g\in C^{1,\alpha}(\R^d)$ be  coercive and convex with modulus of convexity $\omega$, $\Phi\in C^{3,\alpha}(\R^d\backslash\{0\})$ be  uniformly elliptic and  one-homogeneous and let   \(E\)  be a minimizer of either \eqref{e:min} or \eqref{e:minun}. Let  \(S_{\partial E} \) be  the corresponding function defined in \eqref{e:S}. Then, 
\begin{itemize}
\item[(i)] For every \(x\in \partial E\setminus \Sigma\) the point \(\bar y\) achieving the maximum in the definition of \(S_{\partial E}\) is in \( \partial E\setminus \Sigma\).
\item[(ii)] \(S_{\partial E}\in W^{1,2}(\partial E \setminus \Sigma)\).
\item[(iii)] \(S_{\partial E}\) solves 
\begin{equation}\label{e:distr}
L_{\Phi} S_{\partial E}-D_{\nu} g \,S_{\partial E} \ge \omega(S_{\partial E}) \qquad \textrm{on \(\partial E\setminus \Sigma\)}
\end{equation}
both in the viscosity and in the distributional sense.
\end{itemize}
\end{lemma}

\begin{proof}
As above, since the set $E$ is fixed here, we will drop the explicit dependence on $E$ of the various quantities. We divide the proof in few simple steps:

\medskip
\noindent
{\em Step 1: Proof of (i)}. Given \(\bar x\in \partial E\setminus \Sigma\) we let \(\bar y\) be a point achieving the maximum in the definition \eqref{e:S} of \(S\). This means that 
\[
E\subset H=\big\{ y\, :\,  \langle \nu (\bar x), y-\bar x\rangle \le \langle \nu (\bar x), \bar y-\bar x\rangle\big\}\qquad \bar y\in \partial E \cap \partial H. 
\]
Hence  (i) follows for Lemma \ref{lm:reg}.

\medskip
\noindent
{\em Step 2: Proof of (ii)}. Being the supremum of a family of (uniformly)  locally Lipschitz functions, \(S\) is locally Lipschitz on \(\partial E\setminus \Sigma\). Hence, in order to show that \(S\in W^{1,2}(\partial E\setminus \Sigma)\) it is  enough to show that 
\[
\int_{\partial E\setminus \Sigma} |\nabla S|^2 <+\infty.
\]
 Let $\bar x\in \partial E\backslash \Sigma$ and $\bar y$ be such that $S(\bar x)=S(\bar x,\bar y)$, then $S(x,\bar y)\le S(x)$ with equality at $\bar x$ and therefore $\nabla S(\bar x)=\nabla^x S(\bar x,\bar y)$. By \eqref{e:gradS}, we thus have  \(|\nabla S|(\bar x)\le \diam (E)|A|(\bar x)\) and we are left to prove that
 \(|A|\in L^2(\partial E \setminus \Sigma)\). This is a simple consequence of \eqref{e:stab}. Indeed, in case \(E\) minimises  \eqref{e:minun} it is enough to set \(\varphi\equiv 1\) (which is possible thanks to Lemma \ref{lm:cap}) in the second variation inequality \eqref{e:stab} to obtain
\[
\int_{\partial E\setminus \Sigma} |A|^2\le  \int_{\partial E\setminus \Sigma} \tr (D^2 \Phi (\nu)A^2)\le C\int_{\partial E\setminus \Sigma} D_\nu g <+\infty,
\]
where we used the inequality \(|A|^2\le  \tr(D^2 \Phi (\nu) A^2)\) which follows from the ellipticity of \(\Phi\) (recall \eqref{e:unifellipt}). In the  case when \(E\) minimises  \eqref{e:min}, \(\varphi\equiv 1\) is not admissible anymore in \eqref{e:stab}.
However,  letting  \(N_1 \subset \subset N_2\) be two small neighborhoods of the singular set \(\Sigma\), one can   construct two  positive smooth functions \(\varphi_1\) and \(\varphi_2\) such that \(\spt\, \varphi_1\subset N_2\), \(\spt\, \varphi_2\subset \partial E \setminus N_2\), \(\varphi_1\equiv 1\) on \(N_1\)  and
\(\int_{\partial E} \varphi_1=\int_{\partial E} \varphi_2\). Since $\varphi_1-\varphi_2\in W^{1,2}(\partial E\backslash \Sigma)$, by Lemma \ref{lm:cap} we can plug it  in the stability inequality \eqref{e:stab} and deduce that 
\[
\begin{split}
\int_{N_1\backslash \Sigma} |A|^2 &\le   \int_{\partial E\backslash \Sigma} \varphi_1^2 \tr (D^2 \Phi (\nu) A^2)
\\
&\le  -\int_{\partial E\backslash \Sigma} \varphi_2^2 \tr (D^2 \Phi (\nu)A^2)+C\int_{\partial E\backslash \Sigma} \lt( |\nabla \varphi_1|^2+|\nabla \varphi_2|^2\rt) +C\int_{\partial E\backslash \Sigma} |D_\nu g|<+\infty.
\end{split}
\]
Since \( |A|\) is bounded on \(\partial E\setminus N_1\), this concludes the proof. 

\medskip
\noindent
{\em Step 3: Proof of (iii)}. The fact that a viscosity subsolution is a distributional subsolution is proved for instance in \cite[Theorem 1]{Ishii95}.
Hence, it is enough to show that \(S\) satisfies \eqref{e:distr} in the viscosity sense. Thanks to (i), we know  that for every \(\bar x\in \partial E\setminus \Sigma\) 
the point \(\bar y\) achieving the maximum in the definition of \(S\) is in \( \partial E\setminus \Sigma\). Therefore,   Lemma \ref{lm:key} implies that
\[
L_\Phi S (\bar x)\ge H^{\Phi}(\bar x)-H^{\Phi}(\bar y)+\langle \nabla H^{\Phi}(\bar x), \bar y-\bar x\rangle
\]
in the viscosity sense.  Differentiating  \eqref{e:stat} to get $\nabla H^\Phi=-\nabla g$ and subtracting  to both side of the above inequality 
\[
D_\nu g(\bar x) S(\bar x)=D_\nu g(\bar x)\langle \nu(\bar x), \bar y-\bar x\rangle,
\]
 we obtain
\[
\begin{split}
L_\Phi S (\bar x)&-D_\nu g(\bar x) S(\bar x)\\
&\ge H^{\Phi}(\bar x)-H^{\Phi}(\bar y)-\langle \nabla g(\bar x), \bar y-\bar x\rangle-D_\nu g(\bar x)\langle  \nu(\bar x), \bar y-\bar x\rangle
\\
&\stackrel{\eqref{e:stat}}{=}g(\bar y)-g(\bar x)-\langle D g(\bar x), (\bar y-\bar x)\rangle\ge \omega(|\bar x- \bar y|)
\end{split}
\] 
where the last inequality follows by \eqref{e:strictconv}. Since by definition\footnote{If we  knew that $\partial E$ is smooth, then actually $S(\bar x)\le C |\bar x-\bar y|^2$, where the constant  $C$ depends on the curvature and the diameter of $\partial E$.}  $S(\bar x)\le |\bar x-\bar y|$ and since $\omega$ is increasing, this concludes the proof of \eqref{e:distr}.
\end{proof}

We are now ready to prove  our  main results. We start by Theorem \eqref{thm:main2}.

\begin{proof}[Proof of  Theorem \ref{thm:main2}] 
As before, since we work here with a fixed set $E$, we will drop the explicit dependence  on $\partial E$ of the various quantities. Our aim  is to prove that \(S\equiv 0\), which by Lemma \ref{lm:conv} will imply the convexity of \(E\). 

By Lemma \ref{lm:keyupgrade}, \(S\in W^{1,2}(\partial E\setminus \Sigma)\) and thus by Lemma \ref{lm:cap} it can be  approximated  in \(W^{1,2}(\partial E\setminus \Sigma)\) by positive functions in \(C^2_c(\partial E\setminus \Sigma)\). In particular by \eqref{e:stab} (recall \eqref{e:ipp}):
\begin{equation}\label{e:stabS}
\int_ {\partial E\setminus \Sigma } (-L_\Phi S) S+ D_\nu g \,S^2\ge 0.
\end{equation}
Multiplying \eqref{e:distr} by $-S$ we obtain  the inequality 
\[
(- L_{\Phi} S)S + D_{\nu} g S^2\le -\omega(S) S,
\]
which after integration gives 
\begin{equation}\label{e:almostabsurd}
-\int_{\partial E\backslash \Sigma} \omega(S) S\ge\int_ {\partial E\setminus \Sigma }(- L_{\Phi} S)S + D_{\nu} g S^2.
\end{equation}
If $g$ is strictly  convex, this directly gives a contradiction with \eqref{e:stabS} unless $S\equiv 0$. By Lemma \ref{lm:conv}, this implies that $\partial E\subset \partial \co(E)$. 
Now, either by the Constancy Lemma \cite[4.1.31]{Federer69} or by the regularity of \(\partial E\)  (note that  that \(\Sigma=\emptyset\) by Lemma \ref{lm:reg}), \(\partial E=\partial \co(E)\) and hence, since \(E\) is bounded, \(E=\co(E)\).

If instead $g$ is convex but not strictly  convex, we obtain by \eqref{e:almostabsurd} and \eqref{e:stab} (using again Lemma \ref{lm:cap}) that 

\[
\int_ {\partial E\setminus \Sigma }S_{\partial E}(- L_{\Phi} S)S + D_{\nu} g S^2=0\le \min_{\varphi \in W^{1,2}(\partial E \setminus \Sigma)}\int_ {\partial E\setminus \Sigma }(- L_{\Phi} \varphi)\varphi + D_{\nu} g \varphi^2. 
\]
Computing the Euler-Lagrange equation we obtain that 
 \(L_{\Phi} S - D_{\nu} g S\equiv 0\) in \(W^{1,2}(\partial E \setminus \Sigma)\) and then by classical elliptic regularity that $S\in C^2(\partial E\backslash \Sigma)$. Let $M$ be a connected component of $\partial E\backslash \Sigma$ such that $M\cap \partial \co(E)\neq \emptyset$ (which exists by Lemma \ref{lm:reg}). Since $S\ge 0$ on $M$ and
 $S=0$ on $M\cap \partial \co(E)$ by Lemma \ref{lm:conv},  the minimum principle \cite[Theorem 2.10]{hanlin} implies that \(S\equiv 0\) on $M$.   Arguing as above we obtain that $M=\partial \co(E)$, which  in turn gives that $E=\co(E)\backslash F$ for some set $F\subset \subset \co(E)$ and thus
 \[
 \int_{\partial \co(E)} \Phi(\nu_{\partial \co(E)})= \int_{\partial E} \Phi(\nu_{\partial E})- \int_{\partial F} \Phi(\nu_{\partial F})
 \]
 This contradicts the outward minimising property  \eqref{e:minextern} of $E$ unless $F=\emptyset$. Therefore, we can again conclude that $E$ is convex.

\end{proof}

 To prove Theorem \ref{thm:main} we can not plug \(S_{\partial E}\) anymore in the stability inequality since it does not satisfy the zero average constraint.
 Nevertheless, if we assume by contradiction that \(\partial E\) has several  connected components,  then at least one of them must be (unconditionally) stable and this allows to argue as above. In dimension $d>2$ since we can not {\it a priori} 
 guarantee that there exists a connected component which is both stable and intersects $\partial \co(E)$, we need to impose the strict convexity of $g$ to conclude. 
 
 \begin{proof}[Proof of Theorem \ref{thm:main}] As before, we drop the dependence on  $\partial E$ of the various quantities.
 We will actually prove a slightly stronger result with respect to  the connectedness of  \(\partial E\), namely that we can not partition \(\partial E\) as \(\partial E=M_1\cup M_2\) with \(M_1\cap M_2\subset \Sigma\) and \(\HH^{d-1}(M_1), \HH^{d-1}(M_2)>0\). Let us assume for the sake of contradiction that this is not the case. First  we claim that  that there exists \(i\in \{1,2\}\) such that 
 \begin{equation*}
\int_ {\partial M_i\setminus \Sigma } (-L_{\Phi} \varphi)\varphi+ D_\nu g \,\varphi^2\ge 0 \qquad 
\textrm{for all \(\varphi\in C^1_c(M_i \setminus \Sigma) \).}
\end{equation*}
Indeed otherwise, by homogeneity we can find \(\varphi_1\in C^1_c(M_1\setminus \Sigma)\), \(\varphi_2\in C^1_c(M_2\setminus \Sigma)\) such that \(\int_{\partial E} \varphi_1=\int_{\partial E} \varphi_2\) and satisfying 
 \begin{equation}\label{e:onestable}
\int_ {\partial M_i\setminus \Sigma } (-L_{\Phi} \varphi_i)\varphi_i+ D_\nu g \,\varphi_i^2<0
 \qquad \textrm{for \(i=1,2\).}
\end{equation}
Since \(\spt\varphi_1\cap \spt \varphi_2=\emptyset\), the function \(\bar{\varphi}=\varphi_1-\varphi_2\) satisfies \(\int_{\partial E} \bar \varphi=0\) and thus \eqref{e:onestable} would lead to 
 a contradiction with \eqref{e:stab}. Hence either  \(M_1\) or \(M_2\) is stable with respect to all possible variations. For the sake of the argument assume that it is \(M_1\). 
 
\medskip
\noindent
{\em Case 1: \(g\) is strictly convex}: 
 Arguing as in the  proof of  Theorem \ref{thm:main2}, we get (compare   that with \eqref{e:almostabsurd})
 \[
 -\int_{M_1\backslash \Sigma} \omega(S_E) S_E\ge\int_ {M_1\setminus \Sigma }(- L_{\Phi} S)S + D_{\nu} g S^2.
 \]
which by strict convexity of $g$ yields \(S_{E}\equiv 0\) on \(M_1\) and then \(M_1=\partial \co(E)\) with  \(\nu_{\partial E}=\nu_{\partial \co(E)}\) on \(M_1\). Hence, if \(\HH^{d-1}(M_2)>0\),
\[
\int_{\partial \co(E)}\Phi(\nu_{\partial \co(E)})<  \int_{\partial E}\Phi(\nu_{\partial E}),
\]
which contradicts \eqref{e:minextern} and thus \(\HH^{d-1}(M_2)=0\). 

\medskip

{\em Case 2:  \(d=2\) }:  By \cite{McCann98} we know that $E$ is a union of convex sets which are  smooth  by Theorem \ref{thm:promin}. If $ M_1\cap \partial \co(E)\neq \emptyset$, then as in the proof of Theorem \ref{thm:main2} the minimum principle implies that \(S_{E}\equiv0\) on \(M_1\) and thus, by the same arguments as above, $M_1=\partial \co(E)$ and we are done. 

Otherwise, $\partial E\cap \partial \co(E)\subset M_2$. If \(M_2\) is disconnected, by the same arguments as above, at  least one of the connected component must be stable. If this intersects \(\partial \co(E)\) we can repeat the same argument above and conclude. Hence  we can assume that \(M_2\) is connected and unstable,  in particular  $M_2=\partial K$
for some convex set $K$.  We claim that  $\co(E)=K$. Indeed, if this is not the case then there exists an extremal point $x$ of $\partial \co(E)$ which is not in $\partial K$. Since $x\in \partial E$, this contradicts $\partial \co(E) \cap \partial E\subset \partial K$.

\medskip
Connectedness of \(\partial E\) now easily follows both cases: indeed assume that \(\partial E=M_1\cup M_2\) with \(M_1\) and \(M_2\) closed and such that  \(M_1\cap M_2=\emptyset\). Then, we may assume for instance that  \(\HH^{d-1}(M_2)=0\). Thus \(M_2\subset \Sigma\) and \(\partial E \setminus \Sigma \subset M_1\). However regular points are dense in the boundary and thus \(M_2\subset \Sigma \subset \overline{\partial E\setminus \Sigma}\subset M_1\), a contradiction.
\end{proof}

 It is clear that the above proofs mostly rests on the stability inequality \eqref{e:stab} and that minimality is only used to have enough regularity to make the computations in Lemma \ref{lm:key} and Lemma \ref{lm:cap}. The proof can thus be extended to smooth stable critical points
 (or volume preserving stable critical points) of  \(\mathcal F\). For example, we have
 \begin{theorem}\label{thm:stable}
  Let $\Phi$ and $g$ be as in Theorem \ref{thm:main2} and let $E$ be a smooth, bounded and stable critical point of $\mathcal{F}$, then $E$ is convex.
 \end{theorem}
\begin{proof}
 Arguing as in the proof of Theorem \ref{thm:main2}, we obtain that $E=\co(E)\backslash F$ for some smooth set $F$. In order to reach a contradiction, we first claim that $H^\Phi_{\partial E}\ge 0$ (recall the definition \eqref{def:HPhi})  on $\partial E$. Indeed, by \eqref{e:stat}, a minimum $\bar x$ of $H_{\partial E}^{\Phi}$ corresponds to a maximum of $g$ on $\partial E$ and thus,
 \[
  0\ge \tr(D^2\Phi(\nu_{\partial E}) \nabla^2 g(\bar x))\stackrel{\eqref{e:hessian-1}}{=} \tr(D^2\Phi(\nu_{\partial E}) D^2 g(\bar x))-H^{\Phi}_{\partial E}(\bar x) D_\nu g(\bar x).
 \]
From the convexity of $g$, we get that $\tr(D^2\Phi(\nu_{\partial E}) D^2 g(\bar x))\ge 0$. By convexity of $g$ again, $\bar x$ is a maximum of $g$ on $\overline{E}$ and thus $D_\nu g(\bar x)\ge 0$. If the inequality is strict, then $H^\Phi_{\partial E}(\bar x)\ge 0$ as claimed, otherwise since also $\nabla g(\bar x)=0$, we actually have $D g(\bar x)=0$ and $\bar x$
is the minimizer of $g$ on $\R^d$ from which $g$ is constant on $\overline{E}$ and thus by \eqref{e:stat}, $H^\Phi_{\partial E}$ is also constant on $\partial E$. By the anisotropic version of Alexandrov Theorem \cite{he}, this implies that $E$ is actually the Wulff shape of $\Phi$ and is in particular convex. Now if $\bar x\in \partial F$ is such that $H^\Phi_{\partial F}(\bar x)>0$ (which always exists), since $\nu_{\partial E}(\bar x)=-\nu_{\partial F}(\bar x)$, we have $H^{\Phi}_{\partial E}(\bar x)=-H^\Phi_{\partial F}(\bar x)<0$, which gives the desired contradiction.
\end{proof}

 \begin{remark}\label{rmk:stability}
As in the case of minimizers,  it is possible to allow for a small singular set \(\Sigma\) once one knows that the supremum in the definition \(S_{\partial E}\) is achieved by a point \(\bar y\) in the regular set and that Lemma \ref{lm:cap} is in force.  
 
 In the case of the isotropic area functional, Allard's  theorem can be applied to  critical points (see  \cite{Simon83}),  and thus one can extend Lemma \ref{lm:reg} to this setting.  Moreover, using  the  monotonicity formula, Lemma \ref{lm:cap} can also be extended to critical points with a singular set \(\Sigma\) of vanishing \(\HH^{d-3}\) measure. 
 
In the setting of anisotropic surface tensions,  however  both the analog of  Allard's theorem  and the density lower bound \eqref{e:dens} are missing for critical points (see however \cite{Allard86}). 
\end{remark}
 
 \appendix
 
 \section{An approximation Lemma}\label{sec:ap}
 For the reader's convenience we report here the following simple  (and well-known) lemma whose proof follows by a standard capacitary argument (see for instance \cite{De-PhilippisMaggi17,SternZum1}).
 
 \begin{lemma}\label{lm:cap}
 Let \(E\), \(\Phi\), \(g\) and \(\Sigma\) be  as in Theorem \ref{thm:promin}. Then \(C_c^{2}(\partial E\setminus \Sigma)\) is dense in \(W^{1,2}(\partial E\setminus \Sigma)\) with respect to the strong \(W^{1,2}(\partial E \setminus \Sigma)\) topology.
 \end{lemma}
 
 \begin{proof}Obviously  \(\Sigma=\emptyset\) for \(d=2,3\), hence  we assume that \(d\ge 4\).  As already observed in the proof of Theorem \ref{thm:promin}, \(E\) is a \((\Lambda, r_0)\) minimiser of 
 \[
 F\mapsto \int_{\partial^* F}\Phi(\nu_{\partial^* F}).
 \]
 In particular it satisfies the  estimate (see \cite{Maggi12})
 \begin{equation}\label{e:dens}
 \HH^{d-1}(\partial E\cap B(x,r))\le C r^{d-1}\textrm{for all \(x\in \partial E\), \(r\le r_0\)}
 \end{equation}
 for a constant \(C=C(\Lambda,d)>0\). Since \(\HH^{d-3}(\Sigma)=0\) and \(\Sigma\) is compact, for every \(\eps>0\) we can find finitely many balls \(\{B(x_k,r_k)\}_{k=1}^N\) centered in \(\Sigma\) and  such that 
 \[
 \Sigma\subset \bigcup_{k=1}^N B(x_k,r_k)\qquad \sum_{k=1}^N r_{k}^{d-3}\le \eps^{d-3}.
 \]
 For each of this balls we consider \(\varphi_k\in C^{2}_c(B(x_k,2r_k),[0,1])\) satisfying \(\varphi_k\equiv 1\) on \(B(x_k,r_k)\) and \(|D \varphi_k|\le 2/r_k\). Let 
 \[
 \psi_\eps(x)=\max_{k=1,\dots, N} \varphi_k(x),
 \]
 then \(\psi_\eps\in {\rm Lip} (\R^d, [0,1])\), \(\psi=1\) on \(\Sigma\), \(\psi_\eps\equiv 0 \) outside \(\mathcal N_{2\eps}(\Sigma)\), a \(2\eps\) neighborhood of \(\Sigma\). Moreover, by \eqref{e:dens}, 
 \[
 \int_{\partial E} |D \psi_\eps|^2\le 4\sum_{k=1}^N  \frac{\HH^{d-1}(\partial E\cap B(x,2r_k))}{r_k^2}\le C \sum_{k=1}^N r_k^{d-3}\le C\eps^{d-3}.
 \]
 Let now \(u\in W^{1,2}(\partial E \setminus \Sigma)\). By approximation we may assume that \(u\) is bounded and, by scaling that \(\|u\|_\infty\le 1\). Let \(v_\eps=(1-\psi_\eps)u\), then 
 \[
\int_{\partial E}|u-v_\eps|^2+|\nabla u-\nabla v_\eps|^2\le \int_{\mathcal N_{2\eps}(\Sigma)}\lt( u^2+|\nabla u|^2\rt)+2\int_{\partial E}|D\psi_\eps|^2\to 0
 \]
 as \(\eps\to0\). Since \(\spt\,  v_\eps\cap \Sigma=\emptyset\) a simple smoothing argument concludes the proof.
 
 \end{proof}
\bibliographystyle{acm}
\bibliography{convex}

\end{document}